\documentclass[11pt]{amsart}

\usepackage{amssymb,amsfonts}
\usepackage{amsmath,amscd}
\usepackage[all,arc]{xy}
\usepackage{enumerate}
\usepackage{mathrsfs}
\usepackage[pdftex]{graphicx}
\usepackage[utf8]{inputenc}
\usepackage[T1]{fontenc}
\usepackage[usenames,dvipsnames]{color}
\usepackage[bookmarks=false]{hyperref}
\usepackage{dsfont}
\usepackage{tikz}
\usetikzlibrary{automata,positioning}
\usepackage{multirow}

\theoremstyle{plain}
\newtheorem{thm}{Theorem}[section]

\newtheorem{prop}[thm]{Proposition}
\newtheorem{lem}[thm]{Lemma}

\newenvironment{customThm}[1]{\paragraph*{\textbf{Theorem #1.}}\itshape}{\par}
\newenvironment{customProp}[1]{\paragraph*{\textbf{Proposition #1.}}\itshape}{\par}
\newenvironment{customThm1}[2]{\paragraph*{\textbf{Theorem #1} (#2)\textbf{.}}\itshape}{\par}

\theoremstyle{definition}
\newtheorem{defn}[thm]{Definition}

\theoremstyle{remark}

\newcommand{\bbC}{\mathbb{C}} 

\newcommand{\bbH}{\mathbb{H}} 

\newcommand{\bbP}{\mathbb{P}}

\newcommand{\bbR}{\mathbb{R}} 
\newcommand{\bbS}{\mathbb{S}}
\newcommand{\PSL}{\text{PSL}}
\newcommand{\PGL}{\text{PGL}}

\newcommand{\psw}{\text{Pozzetti-Sambarino-Wienhard}}
\newcommand{\bbT}{\mathbb{T}}
\newcommand{\bbZ}{\mathbb{Z}} 
\newcommand*{\defeq}{\mathrel{\vcenter{\baselineskip0.5ex \lineskiplimit0pt
			\hbox{\scriptsize.}\hbox{\scriptsize.}}}%
	=}

\makeatletter
\let\c@equation\c@thm
\makeatother
\numberwithin{equation}{section}

\title{Identities for hyperconvex Anosov representations}

\author{Yan Mary He}
\address{Department of Mathematics, University of Toronto, M5S 2E4, Canada; Math research unit, University of Luxembourg, L-4364, Luxembourg }
\email{yanmary.he@mail.utoronto.ca}
\date{\today}

\begin{document}

\maketitle
\begin{abstract}
In this paper, we establish Basmajian's identity for $(1,1,2)$-hyperconvex Anosov representations from a free group into $\PGL(n,\bbR)$. We then study our series identities on
holomorphic families of Cantor non-conformal repellers associated to complex $(1,1,2)$-hyperconvex Anosov representations. We show that the series is absolutely summable if the Hausdorff dimension of the Cantor set is strictly less
than one. Throughout the domain of convergence, these identities can
be analytically continued and they exhibit nontrivial monodromy.
\end{abstract}

\section{Introduction}
If $\mathcal{C} \subset [0,1]$ is a Cantor set of zero measure, then $a_n$, the
length of the $n$th biggest component of $[0,1]-\mathcal{C}$, satisfy an identity $$1 = \sum_{n=1}^{\infty} a_n.$$ 
For a family of Cantor sets $\mathcal{C}_z \subset \bbC$ depending 
holomorphically on a complex parameter $z$, we obtain by analytic continuation a 
formal holomorphic {\em family} of identities $$S(z) = \sum_{n=1}^{\infty} a_n(z).$$ Thus it is natural to investigate the conditions under which the right-hand-side is absolutely summable.

In \cite{He}, we studied holomorphic families of Cantor sets $\mathcal{C}_z$ associated 
to familiar {\it conformal} dynamical systems, and introduced identities of this form for such families, of which a special case is Basmajian's identity for hyperbolic surfaces. 

In contrast, in this paper, we introduce identities of this form for holomorphic families of Cantor {\it non-conformal} repellers $\mathcal{C}_z$ on the complex projective space $\bbP(\bbC^n)$ which are associated to certain Anosov representations. We show that the right hand sides are absolutely summable if the Hausdorff dimension of $\mathcal{C}_z$ is strictly less than $1$. 

A special case of these identities, where $\mathcal{C}_z \subset \bbP(\bbR^n)$, generalizes Basmajian's identity for hyperbolic manifolds to real hyperconvex Anosov representations.

\subsection{Identities for real $(1,1,2)$-hyperconvex Anosov representations}
If $\Sigma$ is a connected compact oriented hyperbolic surface with geodesic boundary, an
{\em orthogeodesic} $\gamma \subset \Sigma$ is a properly immersed geodesic
arc perpendicular to $\partial \Sigma$. Basmajian (\cite{Bas})
proved the following identity:
\begin{equation}\label{bas_torus}
\text{length}(\partial \Sigma) = \sum_{\gamma} 2 \log \coth^2 \left(\dfrac{\text{length}(\gamma)}{2}\right)
\end{equation}
where the sum is taken over all orthogeodesics $\gamma$ in $\Sigma$.
The geometric meaning of this identity is that there is a canonical decomposition of 
$\partial \Sigma$ into a Cantor set (of zero measure), plus a countable
collection of complementary intervals, one for each orthogeodesic, whose
length depends only on the length of the corresponding orthogeodesic.

For simplicity, one can look at surfaces with a single boundary component.
A hyperbolic structure on $\Sigma$ is the same as a discrete faithful
representation $\rho:\pi_1\Sigma \to \text{PSL}(2,\bbR)$ acting on the upper
half-plane model in the usual way. After conjugation, we may
assume that $\rho(\partial \Sigma)$ stabilizes the positive imaginary axis 
$\ell$. 

Orthogeodesics correspond to double cosets of $\pi_1(\partial \Sigma)$ in
$\pi_1\Sigma$. For each nontrivial double coset $\pi_1(\partial \Sigma) w \pi_1(\partial \Sigma)$, the hyperbolic geodesic $w(\ell)$ corresponds to another boundary component of $\widetilde{\Sigma}$,
and the contribution to Basmajian's identity from this term is
$\log[\infty, 0, \rho(w)(\infty), \rho(w)(0)]$, where $[a,b,c,d]$ is the cross ratio on $\bbR\bbP^1$; hence
$$\text{length}({\partial \Sigma}) = \sum_w \log[\infty, 0, \rho(w)(\infty), \rho(w)(0)].$$

If $n \ge 3$ and $\rho : \pi_1\Sigma \to \PGL(n,\bbR)$ is $(1,1,2)$-hyperconvex in the sense of \psw~\cite{PSW}, then we establish in this paper the following Basmajian's identity:
\begin{equation}\label{eq_realhigher}
\ell_{\rho}(\rho([\partial \Sigma])) = \sum_{w} \log C_{ \rho}(\alpha_1^+,\alpha_1^-,w\cdot \alpha_1^+,w \cdot \alpha_1^-)
\end{equation}
where $\ell_{\rho}$ is a notion of {\it length} with respect to $\rho$, $C_{\rho}$ is a cross ratio defined for four points on the boundary at infinity $\partial\pi_1\Sigma$ of the hyperbolic group $\pi_1\Sigma$ and $\alpha_1^{\pm}$ are the (attracting and repelling) fixed points of $[\partial\Sigma]$ on $\partial\pi_1\Sigma$.
Moreover, if $\rho$ is Hitchin (which is $(1,1,2)$-hyperconvex), equation (\ref{eq_realhigher}) recovers Vlamis-Yarmola's identity \cite{Vlamis}.

\subsection{Complexified identities}
If we deform a real $(1,1,2)$-hyperconvex Anosov representation $\rho: \pi_1\Sigma \to \PGL(n,\bbR)$ by a small deformation to some nearby representation $\rho_z:\pi_1\Sigma \to \text{PGL}(n,\bbC)$,
and replace each $\rho$ by $\rho_z$ above, we obtain a formula for the
{\em complex length} of $\rho_z([\partial \Sigma])$. This is the desired
complexification of our identity (\ref{eq_realhigher}).

Since the right-hand-side of the identity is an infinite series, it would not make sense unless we address the convergence issue. More specifically, we need to find a subspace in the $\PGL(n,\bbC)$-character variety on which the right-hand-series is absolutely convergent. This is where the main technical difficulty lies because the size of the terms in the right-hand-side series are determined by the cross ratios of points in the {\it limit set} of $\rho$, which can be thought of as the repeller of the projective action of $\rho(\pi_1\Sigma)$ on the projective space $\bbP(\bbC^n)$. For Anosov representations, this action on the projective space is non-conformal and it is the lack of conformality that makes it hard to understand the limit set.

In a recent paper (\cite{PSW}), Pozzetti-Sambarino-Wienhard introduced and studied a class of Anosov representations, namely the $(1,1,2)$-hyperconvex Anosov representations. For this class of Anosov representations, the projective action of $\rho(\pi_1\Sigma)$ on $\bbP(\bbC^n)$ admits enough {\it local conformality} and therefore the behavior and the Hausdorff dimension of the limit set of $\rho$ can be understood.

For this reason, we consider the deformation space $\mathcal{S}$ of characters $\rho: \pi_1\Sigma \to \PGL(n,\bbC)$ which are $(1,1,2)$-hyperconvex. 

It turns out that the series makes sense, and is absolutely convergent, throughout
the subset $\mathcal{S}_{<1}$ of $\mathcal{S}$ where the Hausdorff dimension of
the limit set of $\rho_z$ is less than one. Since $\log$ is
multivalued, it is important to choose the correct branch at a real 
representation, and then follow the branch by analytic continuation. The space $\mathcal{S}_{<1}$ is not simply-connected and some terms exhibit monodromy (in the form of integer multiples of $2\pi i$) when they are analytically continued around homotopically nontrivial loops.

\subsection{Statement of results}
The first main theorem of our paper is the following Basmajian's identity for real $(1,1,2)$-hyperconvex Anosov representations.
\vspace{0.3cm}
\begin{customThm}{3.1}
	Let $\Sigma$ be a connected compact oriented hyperbolic surface with geodesic boundaries $a_1,\cdots, a_k$ whose double has genus at least $2$. Let $\alpha_1,\cdots, \alpha_k \in \pi_1\Sigma$ represent the free homotopy classes of $a_1,\cdots, a_k$. If $\rho: \pi_1\Sigma \to \text{PGL}(n, \bbR)$ is $(1,1,2)$-hyperconvex, then
	\begin{equation} \label{eq_algintro}
	\sum_{j=1}^k \ell_{\rho}(\rho(\alpha_j)) = \sum_{j,q =1}^k \sum_{w \in \mathcal{L}_{j,q}} \log C_{ \rho}(\alpha_j^+,\alpha_j^-;w\cdot \alpha_q^+,w \cdot \alpha_q^-) 
	\end{equation}
	where $\alpha_j^+,\alpha_j^-$ are the attracting and repelling fixed points of $\alpha_j$, respectively. Furthermore, if $\rho$ is Hitchin, it is Vlamis-Yarmola's identity.
\end{customThm}
\vspace{0.3cm}

Here $\mathcal{L}_{j,q}$ is a set of double coset representatives associated to boundary components $a_j$ and $a_q$, and $\ell_{\rho}$ and $C_{\rho}$ denote certain length and cross ratio associated to $\rho$ respectively.

Our next theorem concerns the complexification of our identity (\ref{eq_algintro}). More precisely,
\vspace{0.3cm}
\begin{customThm1}{5.1}{Complexified identities}
Let $\Sigma$ be a connected compact oriented hyperbolic surface with geodesic boundaries $a_1,\cdots, a_k$ whose double has genus at least $2$. Let $\alpha_1,\cdots, \alpha_k \in \pi_1\Sigma$ represent the free homotopy classes of $a_1,\cdots, a_k$. If $\rho_0: \pi_1\Sigma \to \PGL(n, \bbR)$ is $(1,1,2)$-hyperconvex and $\rho$ is in the same path component as $\rho_0$ in $\mathcal{S}_{<1}$, then
\begin{equation} \label{eq_alg}
\sum_{j=1}^k \ell_{\rho}(\rho(\alpha_j)) = \sum_{j,q =1}^k \sum_{w \in \mathcal{L}_{j,q}} \log C_{ \rho}(\alpha_j^+,\alpha_j^-;w\cdot \alpha_q^+,w \cdot \alpha_q^-) \text{~~mod~~} 2\pi i,
\end{equation}
where $\alpha_j^+,\alpha_j^-$ are the attracting and repelling fixed points of $\alpha_j$, respectively. Moreover, the series converges absolutely.
\end{customThm1}
\vspace{0.3cm}

Here $\ell_{\rho}$ and $C_{\rho}$ denote the complex length and complex cross ratio respectively, and $\mathcal{S}_{<1}$ is the space of $(1,1,2)$-hyperconvex Anosov representations $\rho: \pi_1\Sigma \to \PGL(n, \bbC)$ whose limit set $\zeta^1(\partial\pi_1\Sigma)$ has Hausdorff dimension strictly smaller than one.

The key ingredient in the proof of the above theorem are the following analytic results.
\vspace{0.3cm}
\begin{customProp}{5.4}
Given a $(1,1,2)$-hyperconvex Anosov representation $\rho: \pi_1\Sigma \to \text{PGL}(n,\bbC)$, the series in (\ref*{eq_algintro}) converges absolutely if the Hausdorff dimension of the limit set $\zeta^{1}(\partial\pi_1\Sigma)$ is strictly less than one.
\end{customProp}
\vspace{0.3cm}

Throughout the domain of convergence, we wish to extend the identity via analytic continuation. To this end, we show that the series is uniformly convergent on compact subsets of $\mathcal{S}_{<1}$ and therefore defines a holomorphic function on $\mathcal{S}_{<1}$.
\vspace{0.3cm}
\begin{customProp}{5.7} 
	Let $\rho_t: \pi_1\Sigma \to \PGL(n,\bbC), t \in [0,1]$ be a deformation of a real $(1,1,2)$-hyperconvex representation $\rho_0$ in $\mathcal{S}_{<1}$. Then the series in (\ref*{eq_algintro}) is uniformly convergent on $[0,1]$.
\end{customProp} 
\vspace{0.3cm}

\subsection{Strategy of proofs}
The proof of our identities for real hyperconvex Anosov representations follows the same decomposition idea as that of Basmajian's identity; that is, we show that there is a canonical decomposition of 
an interval in $\bbR$ into a Cantor set of measure zero, plus a countable
collection of complementary intervals, one for each orthogeodesic, whose
length is given by log of the corresponding cross ratio. In the setting of Anosov representations, our proof adopts the same procedure as Vlamis-Yarmola's proof of Basmajian's identity for Hitchin representations.

Our main analytic result --- the relation between the growth rate of the terms in the right-hand-side of the identities and the Hausdorff dimension --- is proved by using techniques developed by Pozzetti-Sambarino-Wienhard. The main geometric idea is that the size of the terms in the identity can be related to the ratio of singular values $\frac{\sigma_2}{\sigma_1}(\rho(w))$ which are the terms in the Poincar\'e series $$\eta(s) = \sum_{\gamma \in \Gamma} 
\left(\frac{\sigma_2}{\sigma_1}(\rho(\gamma))\right)^s.$$ Pozzetti-Sambarino-Wienhard showed that for $(1,1,2)$-hyperconvex Anosov representations, the critical exponent of $\eta(s)$ equals the Hausdorff dimension of the limit set.

\subsection{$L$-functions and counting results}
In light of Dirichlet's unit theorem, we suggested in \cite{He} that the study of families of Basmajian-type identities is analogous to the idea of studying $L$-functions expressed in series form. In \cite{He1}, we implemented this idea and obtained the first set of results in this direction, namely, the prime number theorems for Basmajian-type identities. In the context of Anosov representations, we expect that one could obtain similar counting and equidistribution results.

\subsection{Organization of the paper} The paper is organized as follows.
In section 2 we set up necessary background. In particular, we define the deformation space and discuss complex cross ratios and complex lengths for Anosov representations. We prove our identities for real $(1,1,2)$-hyperconvex Anosov representations in section 3.
In section 4 we review elements of the Pozzetti-Sambarino-Wienhard paper and introduce the main technical tools and estimates we will need. Finally, in section 5 we prove the main theorem Theorem \ref{thm_cxhigherBas}. In section 6, we discuss monodromy of loops and the topology of $\mathcal{S}_{<1}$.

\subsection{Acknowledgements}
I would like to thank Antonin Guilloux and Andres Sambarino for helpful conversations. I thank Beatrice Pozzetti for pointing out an error in Proposition \ref*{prop_conv}. I thank Binbin Xu for patiently answering my basic questions regarding Anosov representations.

\section{Background}
In this section, we set up necessary background for establishing identities for hyperconvex Anosov representations. In particular, in subsection \ref{subsec_defsp}, we review the basics of Hitchin representations and 
introduce a special class of Anosov representations of a free group into $\PGL(n, \bbC)$, namely the $(1,1,2)$-hyperconvex Anosov representations in the sense of Pozzetti-Sambarino-Wienhard \cite{PSW}. This class of Anosov representations will serve as our deformation space. 

To complexify both sides of the identity, we need to consider complex cross ratios and complex length of any nontrivial element in the fundamental group. We discuss these topics in subsections \ref{subsec_cr} and \ref{subsec_cl} respectively. In the last subsection, we discuss how to represent orthogeodesics algebraically by words in the fundamental group.

\subsection{A space of Anosov representations} \label{subsec_defsp}
Let $\Sigma$ be a connected compact oriented surface with boundary whose double $\hat{\Sigma}$ has genus at least $2$. 

A homomorphism $\rho : \pi_1\Sigma \to \PSL(2,\bbR)$ is {\it Fuchsian} if it is discrete faithful and convex cocompact. A homomorphism $\rho : \pi_1\Sigma \to \PSL(n,\bbR)$ is {\it n-Fuchsian} if $\rho = \iota \circ \rho_0$ where $\iota : \PSL(2,\bbR) \to \PSL(n,\bbR)$ is the irreducible representation and $\rho_0$ is Fuchsian. An element in $\PSL(n,\bbR)$ is {\it purely loxodromic} if all its eigenvalues have multiplicity $1$. Following Labourie and McShane \cite{LabMc}, we say that a homomorphism $\rho : \pi_1\Sigma \to \PSL(n,\bbR)$ is a {\it Hitchin homomorphism} if it can be deformed into an $n$-Fuchsian homomorphism such that the image of each boundary component stays purely loxodromic. A Hitchin representation is the conjugacy class of a Hitchin homomorphism. Labourie-McShane (\cite{LabMc}, Corollary 9.2.2.4) showed that every Hitchin representation $\rho$ for $\Sigma$ is the restriction of a Hitchin representation $\hat{\rho}$ for the closed surface $\hat{\Sigma}$. 

Let $K = \bbR$ or $\bbC$. Given an integer $p$ with $1 \le p \le n-1$, denote by $\mathcal{G}_p(K^n)$ the Grassmannian of $p$-dimensional subspaces of $K^n$.
A homomorphism $\rho : \pi_1\Sigma \to \PGL(n,K)$ is {\it $\{a_p\}$-Anosov} if there exist $\rho$-equivariant H\"older maps $(\zeta_{\rho}^p, \zeta_{\rho}^{n-p}) : \partial \pi_1\Sigma \to \mathcal{G}_p(K^n) \times \mathcal{G}_{n-p}(K^n)$ such that for any $x, y \in \partial\pi_1\Sigma$ with $x \neq y$, we have 
$$\zeta_{\rho}^p(x) \oplus \zeta_{\rho}^{n-p}(y) = K^n,$$ and a suitable associated flow is contracting (see \cite{BCLS}, \cite{Lab06}, \cite{PSW} for details). To simplify notation, we write $\zeta^p = \zeta_{\rho}^p$ if the representation is understood. 
An $a_1$-Anosov representation is also called a projective Anosov representation.

We remark that the {\it boundary at infinity} $\partial\pi_1\Sigma$ of $\pi_1\Sigma$ is defined to be the intersection $\bar{U} \cap \partial\bbH^2$, where $U$ is the universal cover of $\Sigma$. This definition is independent of the choice of a hyperbolic metric on $\Sigma$ since there exists a unique $\pi_1\Sigma$-equivariant H\"older homeomorphism between the limit sets $\rho_1(\pi_1\Sigma)$ and $\rho_2(\pi_1\Sigma)$ where $\rho_1$ and $\rho_2$ are two hyperbolic metrics. This fact also gives $\partial\pi_1\Sigma$ a H\"older structure by embedding it in different ways as a limit set. Moreover, $\partial\pi_1\Sigma$ inherits a metric from $\bbS_{\infty}^1 = \partial\bbH^2$ and the fundamental group $\pi_1\Sigma$ acts as H\"older homeomorphisms on $\partial\pi_1\Sigma$.

Let $p,q,r \in \{1,\cdots, n-1\}$ be integers such that $p+q \le r$. Following Pozzetti-Sambarino-Wienhard (\cite{PSW}), we say that a representation $\rho: \pi_1\hat{\Sigma} \to \PGL(n,K)$ is {\it $(p,q,r)$-hyperconvex} if it is $\{a_p, a_q, a_r\}$-Anosov and for every triple of pairwise distinct points $x,y,z \in \partial\pi_1\Sigma$, we have $$(\zeta^p(x) \oplus \zeta^q(y) ) \cap \zeta^{n-r}(z) = \{0\}.$$ A Hitchin representation $\rho : \pi_1\hat{\Sigma} \to \PSL(n,\bbR)$ is $(p,q,r)$-hyperconvex for any $p+q=r$ (\cite{Lab06} or \cite{PSW} Theorem 9.5).

In this paper, we consider the space $\mathcal{S}$ of $(1,1,2)$-hyperconvex Anosov representations $\rho: \pi_1\Sigma \to \PGL(n,\bbC)$. Note that $\mathcal{S}$ is open in the character variety as the set of $(1,1,2)$-hyperconvex representations is open in Hom$(\pi_1\Sigma, \PGL(n,\bbC))$. If $\rho$ is conjugate to a representation into $\PGL(n,\bbR)$, we call such $\rho$ a real $(1,1,2)$-hyperconvex Anosov representation.

We define the space $\mathcal{S}_{<1}$ to be the set of characters $\rho \in \mathcal{S}$ such that the Hausdorff dimension of the limit set $\zeta^1_{\rho}(\partial\pi_1\Sigma)$ is strictly smaller than one.
We note that $\mathcal{S}_{<1}$ is open in Hom$(\pi_1\Sigma, \PGL(n,\bbC))$ as the set of $(1,1,2)$-hyperconvex representations is open in Hom$(\pi_1\Sigma, \PGL(n,\bbC))$ and the Hausdorff dimension of the limit set is an analytic function on $\mathcal{S}$ (\cite{PSW}, Corollary 1.1).

\subsection{Cross ratios} \label{subsec_cr}
In this subsection, we define a cross ratio for four points on the boundary $\partial\pi_1\Sigma$ which is a natural complexification of Labourie's cross ratio (see \cite{Lab07}, \cite{LabMc}). This is achieved by using the transverse limit maps of a projective Anosov representation which take points of $\partial\pi_1\Sigma$ into the projective spaces, and then taking the cross ratio considered by Falbel-Guilloux-Will \cite{FGW} that we introduce now.

Let $n \ge 3$ be an integer and denote by $\bbC^{n*}$ the dual of $\bbC^{n}$. Consider two non-empty sets $\Omega \subset \bbP (\bbC^{n})$ and $\Lambda \subset \bbP(\bbC^{n*})$ such that $$\varphi(\omega) \neq 0$$ for any $\omega \in \Omega$ and $\varphi \in \Lambda$. In other words, $\Omega$ is disjoint from all hyperplanes defined by points in $\Lambda$.

Let $(\varphi, \varphi', \omega, \omega')\in \Lambda^2 \times \Omega^2$ with $\omega \neq \omega'$ so that $\omega$ and $\omega'$ span a complex projective line $(\omega\omega')$. In \cite{FGW}, Falbel-Guilloux-Will considered the following cross ratio
$$[\varphi, \varphi', \omega, \omega']_{FGW} = [\varphi_{\omega,\omega'}, \varphi'_{\omega,\omega'}, \omega, \omega']$$ where $\varphi_{\omega,\omega'}$ and $\varphi'_{\omega,\omega'}$ are the points $\text{ker}(\varphi) \cap (\omega\omega')$ and $\text{ker}(\varphi') \cap (\omega\omega')$ in $(\omega\omega')$ respectively, and $[a,b,c,d]$ is the usual cross ratio for four distinct points in $\bbC\bbP^1 = (\omega\omega')$, namely $[a,b,c,d] = (a-c)(b-d)/(a-d)(b-c)$.

The FGW-cross ratio satisfies the following properties.
\begin{lem}[\cite{FGW}, Lemma 2.3, Propsition 2.4] \label{lem_cr_FGW}
Given $(\varphi, \varphi', \omega, \omega')\in \Lambda^2 \times \Omega^2$, the cross ratio $[\varphi, \varphi', \omega, \omega']_{FGW}$ satisfies $$[\varphi, \varphi', \omega, \omega']_{FGW} = \frac{\bar{\varphi}(\bar{\omega})\bar{\varphi}'(\bar{\omega}')}{\bar{\varphi}(\bar{\omega}')\bar{\varphi}'(\bar{\omega})}$$
where $\bar{\varphi}, \bar{\varphi}', \bar{\omega}$ and $\bar{\omega'}$ are the lifts of $\varphi, \varphi', \omega$ and $\omega'$. Moreover, if $\omega, \omega', \omega''$ are three points in $\Omega$ and $\varphi, \varphi', \varphi''$ are three points in $\Lambda$, we have
\begin{enumerate}
	\item $[\varphi, \varphi', \omega, \omega']_{FGW} = 1$ if and only if $\varphi = \varphi'$ or $\omega = \omega'$,
	\item $[\varphi, \varphi', \omega, \omega']_{FGW} = [\varphi, \varphi', \omega, \omega'']_{FGW}[\varphi, \varphi', \omega'', \omega']_{FGW}$, and
	\item $[\varphi, \varphi', \omega, \omega']_{FGW} = [\varphi, \varphi', \omega', \omega]_{FGW}^{-1}$.
\end{enumerate}
\end{lem}

Now we define a cross ratio for four points on the boundary $\partial\pi_1\Sigma$. Let $\rho : \pi_1\Sigma \to \PGL(n,\bbC)$ be a $(1,1,2)$-hyperconvex Anosov representation with limit maps $\zeta^1$ and $\zeta^{n-1}$. The {\it complex cross ratio} is a continuous function $C_{\rho} : (\partial\pi_1\Sigma)^{4*} \to \bbR$ given by
$$C_{\rho}(x,y,u,v) = [\zeta^{n-1}(x), \zeta^{n-1}(y), \zeta^{1}(u), \zeta^{1}(v)]_{FGW}.$$
We note that if $\rho : \pi_1\Sigma \to \PSL(n,\bbR)$ is Hitchin, then $C_{\rho}$ is Labourie's cross ratio.

Since $\Sigma$ is oriented, $\partial \pi_1\Sigma$ which can be identified as a subset of $\partial\bbH^2$ inherits a cyclic ordering from the orientation of $\Sigma$. 
For convention, we choose the orientation to be counterclockwise. We say that a quadruple $(x,y,z,w)$ is {\it cyclically ordered} if the triples $(x,y,z)$, $(y,z,w)$ and $(z,w,x)$ are either all positively oriented or negatively oriented.

\begin{lem} \label{lem_pos}
Let $\rho : \pi_1\Sigma \to \PGL(n,\bbR)$ be $(1,1,2)$-hyperconvex. The cross ratio $C_{\rho}$ satisfies the following properties
\begin{enumerate}
\item $C_{\rho}(x,y,u,v) = 0 \iff x = u$ or $y=v$.
\item $C_{\rho}(x,y,u,v) = 1 \iff x = y$ or $u=v$.
\item $C_{\rho}(x,y,u,v) = C_{\rho}(x,y,u,v')C_{\rho}(x,y,v',v)$
\item $C_{\rho}(x,y,u,v) = C_{\rho}(x,y,v,u)^{-1}$
\item For any cyclically ordered quadruple $(x, v, u, y) \in (\partial\pi_1\Sigma)^{4*}$, we have 
$$C_{\rho}(x, y, u, v) > 1.$$
\end{enumerate}
\end{lem}
\begin{proof}
(1)-(4) can be checked directly by definition and Lemma \ref{lem_cr_FGW}.
(5) follows from the fact that hyperconvex Anosov representations are positive.
\end{proof}

\begin{lem}
If $\{\rho_z\}$ is a holomorphic family of projective Anosov representations, then for any fixed four points $x,y,u,v \in \partial\pi_1\Sigma$, the complex cross ratio $C_{\rho_z}(x,y,u,v)$ is holomorphic in $z$.
\end{lem}
\begin{proof}
The lemma follows from the theorem that the limit maps $\zeta^1_{\rho_z}$ and $\zeta^{n-1}_{\rho_z}$ are holomorphic in $z$ (\cite{BCLS}, Theorem 6.1; see also \cite{BPS} section 6).
\end{proof}

\subsection{Complex length and complex period} \label{subsec_cl}
Given a projective Anosov representation $\rho: \pi_1\Sigma \to \PGL(n,\bbC)$, the {\it complex length} of any nontrivial element $\gamma \in \pi_1\Sigma$ is defined, up to $2\pi i$, as
$$\ell_{\rho}(\gamma) = \log \left(\frac{\lambda_1(\rho(\gamma))}{\lambda_{n}(\rho(\gamma))}\right).$$
The {\it complex period} of a nontrivial element $\gamma \in \pi_1\Sigma$ is defined, up to $2\pi i$, as $$P_{\rho}(\gamma) = \log C_{\rho} (\gamma^+,\gamma^-,x,\gamma x )$$ where $x$ is any point in $\partial\pi_1\hat{\Sigma} \setminus \{\gamma^+,\gamma^-\}$.

If $\rho$ is a real projective Anosov representation, then the complex length and complex period become the usual notion of lengths and periods. Moreover, similar to the real case proved by Labourie (\cite{Lab07}, Proposition 5.8), the complex period of a nontrivial element equals its complex length.
\begin{lem}
For any nontrivial element $\gamma \in \pi_1\Sigma$, $P_{\rho}(\gamma) = \ell_{\rho}(\gamma)$.
\end{lem}
\begin{proof}
Straightforward calculation.
\end{proof}

\begin{lem}
If $\{\rho_z\}$ is a holomorphic family of projective Anosov representations, then for a fixed $\gamma \in \pi_1\Sigma$, the complex length function $\ell_{\rho_z}(\gamma)$ is holomorphic in $z$.
\end{lem}
\begin{proof}
Since the eigenvalues $\lambda_1$ and $\lambda_{n}$ of the matrix $\rho_z(\gamma) \in \PGL(n,\bbC)$ varies holomorphically with $z$, $\ell_{\rho_z}(\gamma)$ is holomorphic in $z$.
\end{proof}

\subsection{Orthogeodesics and double cosets}
In this subsection, we show that orthogeodesics on a compact hyperbolic surface with boundary $\Sigma$ correspond to certain double cosets of $\pi_1 \Sigma$. 

Let $\alpha_1,\cdots, \alpha_k \in \pi_1\Sigma$ represent the free homotopy classes of boundary geodesics $a_1,\cdots, a_k$, respectively. Denote $H_j \defeq \langle \alpha_j \rangle$, the subgroup of $\pi_1\Sigma$ generated by $\alpha_j$. Let $\mathcal{DC}(\Sigma)$ denote the set of double cosets of the form $H_pwH_q$ where $w \in \pi_1\Sigma$ is not in $H_p \cap H_q$, for $p,q =1,\cdots,k$. Denote by $[w]_{p,q}$ the class of the double coset $H_p w H_q$.

\begin{prop}[\cite{He}]
	There is a bijection $$\Phi : \mathcal{DC}(\Sigma) \to \{\text{Orthogeodesics on } \Sigma\}.$$
\end{prop}

Let $S$ be a symmetric generating set of $\pi_1\Sigma$. Choose an ordering on $S$. This determines a unique reduced lexicographically first (RedLex) representative $w$ of each double coset $H_p w H_q$. Let $\mathscr{L}_{p,q}$ be the set of nontrivial RedLex double coset representatives for fixed $p, q$. Then the set $$\mathscr{L} \defeq \coprod\limits_{1 \le p, q \le k} \mathscr{L}_{p,q}$$ is naturally in bijection with the set of orthogeodesics on $M$.

\section{Identities for real $(1,1,2)$-hyperconvex Anosov representations}
Let $\Sigma$ be a connected compact oriented surface with boundary whose double $\hat{\Sigma}$ has genus at least $2$. In this section, we generalize Basmajian's identity for hyperbolic manifolds to $(1,1,2)$-hyperconvex Anosov representations $\rho : \pi_1\Sigma \to \PGL(n,\bbR)$. If $\rho$ is Hithcin, we recover Vlamis-Yarmola's identity (\cite{Vlamis}).

\begin{thm} \label{thm_realhigherBas}
	Let $\Sigma$ be a connected compact oriented hyperbolic surface with geodesic boundaries $a_1,\cdots, a_k$ whose double has genus at least $2$. Let $\alpha_1,\cdots, \alpha_k \in \pi_1\Sigma$ represent the free homotopy classes of $a_1,\cdots, a_k$. If $\rho: \pi_1\Sigma \to \text{PGL}(n, \bbR)$ is $(1,1,2)$-hyperconvex, then we have
	\begin{equation} \label{eq_alg}
	\sum_{j=1}^k \ell_{\rho}(\rho(\alpha_j)) = \sum_{j,q =1}^k \sum_{w \in \mathcal{L}_{j,q}} \log C_{ \rho}(\alpha_j^+,\alpha_j^-;w\cdot \alpha_q^+,w \cdot \alpha_q^-) 
	\end{equation}
	where $\alpha_j^+,\alpha_j^-$ are the attracting and repelling fixed points of $\alpha_j$, respectively. Furthermore, if $\rho$ is Hitchin, it is Vlamis-Yarmola's identity.
\end{thm} 

The rest of this section is devoted to the proof of the theorem. The proof follows the same general idea as Basmajian's; namely, we decompose an interval (in $\bbR$) into a set of measure zero which corresponds to the limit set, plus a countable collection of disjoint intervals, one for each orthogeodesic (double coset), whose length is given by log of the corresponding cross ratio. In the context of Anosov representations, our proof adopts the same strategies as Vlamis-Yarmola's proof of Basmajian's identity for Hitchin representations.

\subsection{The limit set has measure zero}
Let $S$ be a closed hyperbolic surface and $\Sigma \subset S$ a connected incompressible subsurface. A hyperbolic structure on $S$ gives an identification of the boundary at infinity $\partial\pi_1S$ with $\bbS^1 = \partial_{\infty}\bbH^2$, under which $\partial\pi_1\Sigma$ has measure zero with respect to the Lebesgue measure on $\bbS^1$. In this subsection, we obtain an analogous result for the limit set associated to a real $(1,1,2)$-hyperconvex Anosov representation of $\pi_1S$.

If $\rho : \pi_1S \to \PGL(n,\bbR)$ is $(1,1,2)$-hyperconvex, then the limit set $\zeta^1(\partial\pi_1S)$ is a $C^{1+\alpha}$-submanifold of $\bbP(\bbR^n)$ (\cite{ZhangZimmer}, Theorem 1.1). Let $\eta_{\rho} : \bbS^1 \to \zeta^1(\partial\pi_1S)$ be a $C^1$-parametrization with $\alpha$-H\"older derivative and let $\lambda$ be the Lebesgue measure on $\bbS^1$. We define a measure $\mu_{\rho}$ on $\partial\pi_1S$ by setting
$$\mu_{\rho} = ((\zeta^1)^{-1}\circ \eta_{\rho})_* \lambda.$$

\begin{prop}\label{prop_mzero}
Let $S$ be a closed hyperbolic surface and $\Sigma \subset S$ a connected incompressible subsurface. Let $\rho$ be a real $(1,1,2)$-hyperconvex Anosov representation of $\pi_1S$. If $\mu_{\rho}$ is the pushforward of the Lebesgue measure on $\partial\pi_1S$, then viewing $\partial\pi_1\Sigma \subset \partial\pi_1S$, we have $$\mu_{\rho}(\partial\pi_1\Sigma) = 0.$$
\end{prop}
\begin{proof}
Vlamis-Yarmola's proof (\cite{Vlamis}, section 3) works verbatim here. The property of the limit set used in their proof is that it is a $C^{1+\alpha}$-submanifold of $\bbP(\bbR^n)$. This is also the case for real $(1,1,2)$-hyperconvex representations.
\end{proof}

\subsection{Decomposition: proof of the theorem}
Recall that for a hyperbolic surface $S$ with or without boundary, $\partial\pi_1S$ inherits an ordering from $\partial_{\infty}\bbH^2$. For any $u,v \in \partial\pi_1S$, we define 
$$(u,v) = \{w \in \partial\pi_1S ~|~ (u,w,v) \text{ is positive} \}$$
Given a loxodromic element $\alpha \in \pi_1S$ with fixed points $\alpha^{\pm} \in \partial\pi_1S$ and a reference point $y \in (\alpha^+, \alpha^-) \subset \pi_1S$, we define a function $F_{\rho} : (\alpha^+, \alpha^-) \to \bbR$ by
$$F_{\rho} (x) = \log C_{\rho}(\alpha^+, \alpha^-; x, y).$$
Note that $C_{\rho}(\alpha^+, \alpha^-; x, y)$ is positive by Lemma \ref{lem_pos}.

\begin{lem}
The map $F_{\rho}$ is a homeomorphism onto its image and it is surjective if $S$ is closed.
\end{lem}
\begin{proof}
Our proof follows the same lines as the proof of Lemma 6.1 in \cite{Vlamis}.
We first show $F_{\rho}$ is injective. If $C_{\rho}(\alpha^+, \alpha^-, x, y) = C_{\rho}(\alpha^+, \alpha^-, x', y)$, then by (3) of Lemma \ref{lem_pos}, $C_{\rho}(\alpha^+, \alpha^-, x, x') = 1$ which implies $x = x'$ by (2) of Lemma \ref{lem_pos}. Moreover, (5) of Lemma \ref{lem_pos} implies $F_{\rho}$ preserves the ordering and therefore it is a homeomorphism onto its image. Finally, $F_{\rho}(x) \to \mp\infty$ as $x \to \alpha^{\pm}$ and $(\alpha^+, \alpha^-)$ is connected if $S$ is closed.
\end{proof}

Now we give the proof of Theorem \ref{thm_realhigherBas}.
\begin{proof}[Proof of the Theorem 3.1] 
The proof follows the same steps as that in Vlamis-Yarmola (\cite{Vlamis}, section 6). We give a sketch of the proof for the convenience of the reader and more details can be found in \cite{Vlamis}, section 6.

We fix a boundary component $\alpha_j$. Let $U \subset \bbH^2$ be the universal cover of $\Sigma$ so that $\partial\pi_1\Sigma = \partial_{\infty}U = \bar{U} \cap \bbS^1$. Then $\bbS^1 \setminus \partial_{\infty}U$ is a {\it disjoint} union of intervals of the form $\tilde{I}_{\beta} = (\beta^-, \beta^+)$ such that there is the following one-to-one correspondence $$\{ \text{Components } \tilde{I}_{\beta} \text{ of } \bbS^1 \setminus \partial_{\infty}U \} \longleftrightarrow \bigsqcup_{1 \le q \le k} \pi_1\Sigma / H_q.$$ This correspondence can be seen by looking at the action of $\pi_1\Sigma$ on $\bbH^2$.

Since $F_{\rho} : (\alpha_j^+, \alpha_j^-) \to \bbR$ is increasing, $\bbR \setminus F_{\rho}(\partial_{\infty}U)$ is a disjoint union of intervals $\hat{I}_\beta = (F_{\rho}(\beta^-), F_{\rho}(\beta^+))$. Furthermore, we have $$F_{\rho}(\alpha_j \cdot x) = F_{\rho}(x) - \ell_{\rho}(\alpha_j).$$

Set $\bbT = \bbR/\ell_{\rho}(\alpha_j)\bbZ$ and let $\pi : \bbR \to \bbT$ be the projection map. Define $I_{\beta} = \pi(\hat{I}_{\beta})$ and we have the following one-to-one correspondence
$$\{ \text{Components } I_{\beta} \text{ of } \bbT \setminus \pi(F_{\rho}(\partial_{\infty}U)) \} \longleftrightarrow \bigsqcup_{1 \le q \le k, q \neq j} H_j \backslash \pi_1\Sigma / H_q.$$

If $\lambda$ is the Lebesgue measure on $\bbR$, then $$\lambda(I_{\beta}) = F_{\rho}(\beta^+) - F_{\rho}(\beta^-) = \log C_{\rho}(\alpha_j^+, \alpha_j^-; w \cdot \alpha_j^+, w \cdot \alpha_j^-)$$
where $w$ is a double coset representative corresponding to $I_{\beta}$.

Therefore, for boundary $\alpha_j$, we have 
$$\ell_{\rho}(\alpha_j)  = \lambda(\bbT) = \lambda(\pi(F_{\rho}(\partial_{\infty}U)))  + \sum_{\beta} \lambda(I_{\beta})$$
where $\lambda(\pi(F_{\rho}(\partial_{\infty}U))) = 0$ by Proposition \ref{prop_mzero} and the fact that $\pi \circ F_{\rho} = \pi \circ \log C_{\rho}$ is differentiable.
The identity is obtained by summing over $\alpha_j$'s.
\end{proof}

\section{Hausdorff dimension of limit sets}
In this section, we summarize tools and estimates from Pozzetti-Sambarino-Wienhard's paper \cite{PSW} that we will need to prove the convergence criterion in the next section.

Let $\Gamma$ be a hyperbolic group and consider a $(1,1,2)$-hyperconvex Anosov representation $\rho : \Gamma \to \PGL(n,\bbC)$. One of the main theorems in their paper is to show that the Hausdorff dimension of the limit set equals the critical exponent of some Poincar\'e series, which is an analogue of Sullivan's theorem in the context of Kleinian groups.

\begin{thm}[\cite{PSW}]
We have $$h_{\rho}^{a_1} = \text{H.dim}(\zeta^1(\partial \Gamma) \le \text{H.dim}(\bbP(\bbC^2))=2$$
where $h_{\rho}^{a_1}$ is the critical exponent of the Poincar\'e series 
\begin{equation}\label{eq_poincare}
\eta(s) = \sum_{\gamma \in \Gamma} \left(\frac{\sigma_2}{\sigma_1}\rho(\gamma)\right)^s
\end{equation}
and $\sigma_i$'s are the singular values of $\rho(\gamma)$.
\end{thm}

We start by recalling the definition of distance on projective spaces.
\subsection{Distances on Grassmannians} \label{subsec_dgrass}
Let $V$ be a finite dimensional vector space over $\bbC$ and denote by $G_k(V)$ the Grassmannian of $k$-planes in $V$. Then for $P,Q \in G_k(V)$, we define the {\it distance} $d$ on $G_k(V)$ by
$$d(P,Q) = \max_{v \in P^{*}} \min_{w \in Q^*} \sin \angle(v,w)$$
where $P^{*} = P \setminus \{0\}$, $Q^{*} = Q \setminus \{0\}$ and $\angle(v,w)$ is the unique number in $[0,\pi]$ such that $\sin \angle(v,w) = \frac{||v \wedge w||}{||v||\cdot ||w||}$.

\subsection{Cone types at infinity} \label{subsec_conetype}
Given $\gamma \in \Gamma$, the {\it cone type} of $\gamma \in \Gamma$ is defined by Cannon as $$\mathcal{C}(\gamma) = \{\eta \in \Gamma ~|~ |\gamma\eta| = |\eta| + |\gamma|  \}.$$

If $\Gamma$ is a free group of rank $n$, then $\Gamma$ has $2n+1$ cone types, i.e. $\mathcal{C}(e)$ and $\mathcal{C}(g_i^{\pm})$ for $i = 1, \cdots, n$. Moreover, if $e \neq \gamma \in \Gamma$, then $\mathcal{C}(\gamma) = \mathcal{C}(g_i)$ (or $\mathcal{C}(\gamma) = \mathcal{C}(g_i^{-1})$) if the last letter of $\gamma$ is $g_i$ (or $g_i^{-1}$).

We associate to every cone type $\mathcal{C}(\gamma)$ a subset of $\partial \Gamma$, the {\it cone type at infinity}, by considering limit points of geodesic rays starting at the identity and contained entirely in $\mathcal{C}(\gamma)$, i.e.
$$\mathcal{C}_{\infty}(\gamma) = \{ [(\alpha_i)] ~|~ (\alpha_i) \text{ is a geodesic ray with } \alpha_0 = e, \alpha_i \in \mathcal{C}(\gamma)  \}.$$
Since there are only finitely many cone types, we have a finite cover of $\partial \Gamma$ by $\mathcal{C}_{\infty}(\gamma)$.

With cone types at infinity, we consider the following cover of $\partial \Gamma$ which serves as Sullivan's shadows.
\begin{lem}  [\cite{PSW}, Lemma 2.3]\label{lem_cover} 
Given $T>0$, the family of open sets $$U_T = \{\gamma\mathcal{C}_{\infty}(\gamma) ~|~ |\gamma| \ge T\}$$ defines an open cover of $\partial \Gamma$.
\end{lem}

The following lemma gives an upper bound for the diameter of the set $\zeta^1(\gamma\mathcal{C}_{\infty}(\gamma))$ in terms of the ratio of singular values of $\rho(\gamma)$. As a consequence, the critical exponent $h_{\rho}^{a_1}$ is an upper bound for the Hausdorff dimension.
\begin{lem} \label{lem_diam}
Let $\rho : \Gamma \to \PGL(n,\bbC)$ be projective Anosov. Then there exists $\delta>0$ such that $$\text{diam} \zeta^1(\gamma\mathcal{C}_{\infty}(\gamma)) \le \frac{2}{\delta} \frac{\sigma_2}{\sigma_1}(\rho(\gamma)).$$
\end{lem}
\begin{proof}
The proof is given in Lemma 4.2 and the proof of Proposition 4.1 in \cite{PSW}.
\end{proof}

\subsection{Locally conformal points and thickened cone types at infinity}
To show that the critical exponent $h_{\rho}^{a_1}$ is also a lower bound for the Hausdorff dimension, Pozzetti-Sambarino-Wienhard introduced {\it locally conformal points} in the boundary of the group $\partial\Gamma$ which detects local conformality of the projective action.

\begin{defn}
Let $\rho : \Gamma \to \PGL(n,\bbC)$ be projective Anosov. We say that $x \in \partial\Gamma$ is a {\it $(\varepsilon, L)$-locally conformal point} of $\rho$ if there exists a geodesic ray $\{\alpha_i \}$ in $\Gamma$ based at the identity with end point $x$ such that
\begin{enumerate}
	\item for all big enough $i$, $p_2(\alpha_i) = p_2$ does not depend on $i$,
	\item for every $i > L$ and for every $z \in \mathcal{C}_{\infty}(\alpha_i)$, one has $$\sin \angle(\zeta^1(z) \oplus \rho(\alpha_i^{-1})\zeta^1(x), U_{d-p_2}(\rho(\alpha_i^{-1})) > \varepsilon.$$
\end{enumerate}
where $p_2(\alpha_i)$ is the index of the second gap of the eigenvalues of $\alpha_i$.
\end{defn}

To prove the lower bound on Hausdorff dimension, the authors considered coverings of $\zeta^1(\partial\Gamma)$ obtained by translating {\it thickened cone types at infinity} that we define now.

Let $\rho : \Gamma \to \PGL(n,\bbC)$ be projective Anosov. Since there is a positive lower bound on the distance of attractors and repellers of geodesic rays through the origin, define the {\it least angle} $\delta_{\rho}$ by $$\delta_{\rho} = \inf \sin \left( \angle(U_1(\rho(\alpha_k)),  U_{d-1}(\rho(\alpha_{-m}))\right)$$
where $(\alpha_i)$ is a biinfinite geodesic through the origin, and $k,m > L$ for $L$ large.
The {\it thickened cone type at infinity} $X_{\infty}(\gamma)$ is a $\delta_{\rho}/2$-neighbourhood of $\zeta^1_{\rho}(\mathcal{C}_{\infty}(\gamma))$ in $\zeta^1_{\rho}(\partial \Gamma)$, i.e.
$$X_{\infty}(\gamma) = N_{\delta_{\rho}/2} (\zeta^1_{\rho}(\mathcal{C}_{\infty}(\gamma))) \cap \zeta^1_{\rho}(\partial \Gamma).$$

The next proposition provides an estimate for the distance between two points in the projective space which will be useful to us. In our case, $x$ and $y$ in the following proposition will play the role of translates of the two fixed points of a boundary element in $\partial\Gamma$.

\begin{prop}[\cite{PSW}, Proposition 5.9]\label{prop_5.9}
There exists a constant $c_1$ depending only on $\rho$ such that if $L$ is big enough and $\{g_i\} \subset \Gamma$ is a geodesic ray with endpoint $x$ which is locally conformal, then for every $y \in \partial\Gamma$ such that $\zeta^1_{\rho}(y) \in \rho(g_n)X_{\infty}(g_n) \setminus \rho(g_{n+L})X_{\infty}(g_{n+L})$, we have $$d(\zeta^1_{\rho}(y), \zeta^1_{\rho}(x)) \ge c_1 \frac{\sigma_2}{\sigma_1}(\rho(g_{n+L})).$$ 
\end{prop}

Following this proposition, the authors showed that the critical exponent is also a lower bound for the Hausdorff dimension of the limit set, provided the existence of ``abundant'' local conformal points.

\subsection{Hyperconvexity and local conformality}
The link between hyperconvexity and local conformality is provided in the following proposition, which states that hyperconvexity guarantees the abundance of local conformal points.

\begin{prop}[\cite{PSW}, Proposition 6.7] \label{prop_hypconf}
Let $\rho : \Gamma \to \PGL(n, \bbC)$ be $(p,q,r)$-hyperconvex. Then there exist constants $L, \varepsilon$ such that for every $\alpha \in \Gamma$ with $|\alpha| > L$, for every $x \in \mathcal{C}_{\infty}(\alpha)$ and every $y \in (\zeta^q)^{-1}X_{\infty}(\alpha)$, it holds
$$\sin \angle(\zeta^p(x)\oplus \zeta^q(y), U_{d-r}(\rho(\alpha^{-1})) > \varepsilon.$$
\end{prop}

\section{Complexified identities}
The goal of this section is to prove the following complexification of our identities.
\begin{thm}[Complexified identities] \label{thm_cxhigherBas}
	Let $\Sigma$ be a connected compact oriented hyperbolic surface with geodesic boundaries $a_1,\cdots, a_k$ whose double has genus at least $2$. Let $\alpha_1,\cdots, \alpha_k \in \pi_1\Sigma$ represent the free homotopy classes of $a_1,\cdots, a_k$. If $\rho_0: \pi_1\Sigma \to \text{PGL}(n, \bbR)$ is a real $(1,1,2)$-hyperconvex Anosov representation and $\rho$ is in the same path component as $\rho_0$ in $\mathcal{S}_{<1}$, then
	\begin{equation} \label{eq_alg}
	\sum_{j=1}^k \ell_{\rho}(\rho(\alpha_j)) = \sum_{j,q =1}^k \sum_{w \in \mathcal{L}_{j,q}} \log C_{ \rho}(\alpha_j^+,\alpha_j^-;w\cdot \alpha_q^+,w \cdot \alpha_q^-) \text{~~mod~~} 2\pi i,
	\end{equation}
	where $\alpha_j^+,\alpha_j^-$ are the attracting and repelling fixed points of $\alpha_j$, respectively. Moreover, the series converges absolutely.
\end{thm} 
We obtain the complexified identity via analytic continuation as we deform a real $(1,1,2)$-hyperconvex Anosov representation into the $\PGL(n,\bbC)$-character variety. 

The main technical problem for analytic continuation to work is the convergence issue of the right-hand-side series in our identity. It turns out that for a $(1,1,2)$-hyperconvex Anosov representations $\rho : \pi_1\Sigma \to \PGL(n,\bbC)$, the right-hand-side series is absolutely convergent if the Hausdorff dimension of the limit set $\zeta^1_{\rho}(\partial\pi_1\Sigma)$ is strictly smaller than one. Denote by $\mathcal{S}_{<1}$ the subspace of $\rho \in \mathcal{S}$ such that the Hausdorff dimension of the limit set is strictly smaller than one. We prove that the right-hand-side series is uniformly convergent on compact subsets of $\mathcal{S}_{<1}$ and therefore defines a holomorphic function on $\mathcal{S}_{<1}$.

The proof of absolute convergence of the series uses techniques developed by Pozzetti-Sambarino-Wienhard. We show that if the double coset representative $w$ is a long word, then the absolute value of the term associated to $w$ in the series is bounded above by the diameter of $\zeta^1(w\mathcal{C}_{\infty}(w) )$, which is bounded above by some constant multiple of $\frac{\sigma_2}{\sigma_1}\rho(w)$ by Pozzetti-Sambarino-Wienhard.

\subsection{The left-hand-side of the identity}
First we show that the left-hand-side of the identity behaves well under deformation.
\begin{lem}\label{lem_lhs}
	Let $\rho_0: \pi_1\Sigma \to \PGL(n, \bbR)$ be a real $(1,1,2)$-hyperconvex Anosov representation with $\ell_{\rho}(\rho(g))$ real positive for all $g \in \pi_1\Sigma$. Then Re$(\ell_{\rho}(\rho_t(g)))>0$ for all $g \in \pi_1\Sigma$ when analytically continued along a path $\rho_t$ in $\mathcal{S}$.
\end{lem}
\begin{proof}
	If there were some $t$ and some $g \in \pi_1\Sigma$ such that Re$(\ell_{\rho}(\rho_t(g)))=0$, then it would contradict the fact that the representation is projective Anosov.
\end{proof}

\subsection{Convergence theorems}
In this subsection, we first give a convergence criterion which gives us a subset of $\mathcal{S}$ on which the right-hand-side series makes sense. Our criterion states that the series in (\ref*{eq_alg}) converges absolutely if the Hausdorff dimension of the limit set $\zeta^{1}(\partial\pi_1\Sigma)$ is strictly less than one. The idea of the proof is to show that the tail of our series is bounded above by some constant multiple of the tail of the Poincar\'e series $\eta(s)$ as in (\ref{eq_poincare}). Then we prove that the right-hand-side series is uniformly continuous on compact subsets of $\mathcal{S}_{<1}$.

\begin{prop}[Convergence criterion] \label{prop_conv}
If $\rho: \pi_1\Sigma \to \text{PGL}(n,\bbC)$ is a $(1,1,2)$-hyperconvex Anosov representation, the series in (\ref*{eq_alg}) converges absolutely if the Hausdorff dimension of the limit set $\zeta^{1}(\partial\pi_1\Sigma)$ is strictly less than one.
\end{prop}

The following two lemmas constitute the proof of Proposition \ref{prop_conv}. We start with some notations. Fix two integers $j, q \in \{1, \cdots, k\}$. Denote by $\varphi = \zeta^{n-1}_{\rho}(\alpha_j^+), \varphi' = \zeta^{n-1}_{\rho}(\alpha_j^-), w^+ = \omega = \zeta^1_{\rho}(w \cdot \alpha_q^+)$ and $w^- = \omega' = \zeta^1_{\rho}(w \cdot \alpha_q^-)$. Then by definition, $$C_{\rho}(\alpha_j^+,\alpha_j^-;w\cdot \alpha_q^+,w \cdot \alpha_q^-) =[\varphi, \varphi', w^+, w^-]_{FGW}.$$

Since $\varphi \neq \varphi'$, we conjugate the representation and choose a basis $\{e_m\}_{m=1}^{n}$ of $\bbC^n$ such that the lifts (in $\bbC^{n*}$ and $\bbC^{n}$) $\bar{\varphi}, \bar{\varphi}', \bar{\varphi}_{\omega,\omega'}$ and $\bar{\varphi}'_{\omega,\omega'}$ are as follows:
$$\bar{\varphi} = e_1^*, \bar{\varphi}' = e_2^*, \bar{\varphi}_{\omega,\omega'} = e_2 + \bar{u}, \bar{\varphi}'_{\omega,\omega'} = e_1 + \bar{u}'$$
where $\bar{u}$ and $\bar{u}'$ are vectors in Span$(e_3, \cdots, e_n)$. Note that $\bar{\varphi}_{\omega,\omega'}$ and $\bar{\varphi}'_{\omega,\omega'}$ form a basis for the plane spanned by $\bar{\omega}$ and $\bar{\omega}'$.

Therefore $\bar{\omega}$ and $\bar{\omega}'$ can be written as
$$\bar{\omega} =  \lambda'\bar{\varphi}'_{\omega,\omega'} + \lambda\bar{\varphi}_{\omega,\omega'} \text{ and } \bar{\omega}' = \mu\bar{\varphi}'_{\omega,\omega'} + \mu'\bar{\varphi}'_{\omega,\omega'}.$$

With respect to the coordinate basis $\{\bar{\varphi}'_{\omega,\omega'}, \bar{\varphi}_{\omega,\omega'}\}$, in the projective line $(\omega\omega')$, $\varphi_{\omega,\omega'}$ has projective coordinate $[1, \infty = 1/0]$, $\varphi'_{\omega,\omega'} = [1, 0]$, $w^+ = \omega = [\lambda',\lambda] = [1; \lambda/\lambda']$ and $w^+ = \omega = [\mu',\mu] = [1; \mu/\mu']$.

Therefore,
$$C_{\rho}(\alpha_j^+,\alpha_j^-;w\cdot \alpha_q^+,w \cdot \alpha_q^-) = [\varphi_{\omega,\omega'}, \varphi'_{\omega,\omega'}, w^+, w^-] = \frac{w^-}{w^+}$$
where $w^+ = \lambda/\lambda'$ and $w^- = \mu/ \mu'$.

\begin{lem}\label{lem_logd}
	Take the principal branch of $\log$. If $w \in \mathscr{L}_{j,q}$ with $|w|$ large, then
	$$\biggr|\log \dfrac{w^-}{w^+}\biggr| \asymp d(w^+, w^-)$$ where $d$ is the distance on $\bbP(\bbC^n)$ (see subsection \ref{subsec_dgrass}).
\end{lem}

\begin{proof}
Since $\log(1+z) = \sum_{n=1}^{\infty}\dfrac{(-1)^{n+1}}{n}z^n$ for $|z|<1$, for $|z|$ small, $|\log(1+z)| \asymp |z|$. Therefore, for $|w|$ large, we have
	$$\biggr|\log \dfrac{w^-}{w^+}\biggr| = \biggr\rvert \log \left(1+ \dfrac{w^--w^+}{w^+}\right)\biggr\rvert \asymp \biggr| \dfrac{w^--w^+}{w^+}\biggr|.$$
The term $|w^+|$ is uniformly bounded for all $w$. Indeed, for each $w$, $|w^+|$ is bounded above and below as there is the least angle between repellers and attractors so that $w^+$ and $w^-$ can not be arbitrarily close to $\bar{\varphi}_{\omega,\omega'}$ and $\bar{\varphi}'_{\omega,\omega'}$, and the limit set is compact.

On the other hand, the projective distance $d(w^+, w^-)$ is comparable to the angle between $w^+$ and $w^-$ since $\sin x \asymp x$ for $x$ small. Moreover, on the projective line, the angle between $w^+$ and $w^-$ is comparable to $\left|w^+-w^-\right|$.
\end{proof}

The following lemma proves Proposition \ref{prop_conv}.
\begin{lem} \label{lem_subseries_1}
For each $j,q = 1, \cdots, k$, the series $$\displaystyle\sum_{w \in \mathscr{L}_{j,q}} w^+ - w^-$$ is absolutely convergent if the Hausdorff dimension of the limit set is strictly less than one.
\end{lem}
\begin{proof} 
For any word $w\in \mathscr{L}_{j,q}$, by definition of cone type at infinity $\mathcal{C}_{\infty}(w)$, $w \cdot \alpha_q^+$ and $w \cdot \alpha_q^-$ are contained in $w\mathcal{C}_{\infty}(w)$. Then $w^+ = \zeta^1_{\rho}(w \cdot \alpha_q^+)$ and $w^- = \zeta^1_{\rho}(w \cdot \alpha_q^-)$ are contained in the set $\zeta^1_{\rho}(w\mathcal{C}_{\infty}(w)) \subset \bbP(\bbC^n)$. 

Then for $w$ with $|w| \ge T$ where $T$ is given in Lemma \ref{lem_cover}, we have
$$d(w^+, w^-) \le \text{diam}\zeta^1(w\mathcal{C}_{\infty}(w)) \le c_1\frac{\sigma_2}{\sigma_1}(\rho(w))$$
for some constant $c_1>0$. The last inequality follows from Lemma \ref{lem_diam}. 

Then it follows that $$\sum_{|w|\ge T} \left|\log \frac{w^-}{w^+}\right| \le c\sum_{|w|\ge T} d(w^+, w^-) \le cc_1\sum_{|w|\ge T}\frac{\sigma_2}{\sigma_1}(\rho(w)).$$
Therefore, if the Hausdorff dimension is strictly less than one, $\sum_{|w|\ge T}\frac{\sigma_2}{\sigma_1}(\rho(w))$ is finite and therefore the series is absolutely convergent.
\end{proof}

Next, we prove that the right-hand-side series is uniformly convergent on compact subsets of $\mathcal{S}_{<1}$.

\begin{prop} [Uniform convergence] \label{cor_unif_conv}
	Let $\rho_t: \pi_1\Sigma \to \PGL(n,\bbC), t \in [0,1]$ be a deformation of a real $(1,1,2)$-hyperconvex Anosov representation $\rho_0$ in $\mathcal{S}_{<1}$. Then the series in (\ref*{eq_alg}) is uniformly convergent on $[0,1]$.
\end{prop} 
\begin{proof}
	The proposition follows from the proof of the convergence theorem. Since the absolute series is uniformly bounded above by the Poincar\'e series $\sum_{\gamma \in \pi_1\Sigma} \frac{\sigma_2}{\sigma_1}(\rho_t(\gamma))$, which is continuous on $[0,1]$, it is uniformly bounded. A uniformly bounded sequence of holomorphic functions which converges pointwise (by Proposition \ref{prop_conv}) converges uniformly.
\end{proof}

\subsection{Analytic continuation}
\begin{proof}[Proof of Theorem \ref{thm_cxhigherBas}]
	Let $\rho_t$ in $\mathcal{S}_{<1}$ be an analytic deformation of $\rho_0$. By Lemma \ref{lem_lhs}, the analyticity (in $\rho$) of limit maps $\zeta_{\rho}^1, \zeta_{\rho}^{n-1}$ (see \cite{BCLS}, section 6 or \cite{BPS} section 6) and Proposition \ref{cor_unif_conv}, each side of identity \ref*{eq_alg} is a holomorphic function on $\mathcal{S}_{<1}$ (up to $2\pi i$). In a neighbourhood $U$ of $\rho_0$, the identity holds on an open subset consisting of real $(1,1,2)$-hyperconvex Anosov representations. Therefore it holds on the entire neighbourhood $U$. Hence, by analytic continuation, the identity holds for all $\rho_t$ (up to $2\pi i$).
\end{proof}

\section{Monodromy and topology of the character variety}
In this section, we discuss one application of our identities to the topology of character varieties. 
As we analytically continue (\ref*{eq_alg}) along a loop in $\mathcal{S}_{<1}$ the value of the right-hand-side might change by multiples of $2\pi i$. We call this the {\it monodromy} of the loop and observe that it depends only on its homotopy class in $\mathcal{S}_{<1}$. Note that the uniform convergence throughout a compact loop in $\mathcal{S}_{<1}$ implies that only finitely many words can change monodromy. We will exhibit two loops in $\mathcal{S}_{<1}$ that are not homotopically equivalent.

Let $F_2 = \langle a,b \rangle$ be a free group on two generators with $A = a^{-1}$ and $B = b^{-1}$. Let $L,x \in \bbC$ be such that $x+1/x = -L$. Consider the Schottky groups $\Gamma_L \defeq \rho(\langle a,b \rangle) = \langle X_L, Y_L \rangle$ and $\Gamma'_L \defeq \rho'(\langle a,b \rangle) = \langle X_L^2, X_LY_L^3 \rangle$, where
$$X_L = \left[ \begin{array}{cc}
L & 1 \\
-1 & 0 \end{array} \right] \text{ , } 
Y_L = \left[ \begin{array}{cc}
0 & x \\
-\frac{1}{x} & L \end{array} \right].$$

In \cite{He}, we exhibited two different loops in the $\text{PSL}(2,\bbC)$-character variety of $F_2$, $\Gamma_{L_t}$ and $\Gamma'_{L_t}$ by letting $$L_t = 5e^{2\pi it}, t \in [0,1]$$ such that $\Gamma_{L_0} = \Gamma_{L_1}$ and $\Gamma'_{L_0} = \Gamma'_{L_1}$ are Fuchsian, and for each $t \in [0,1]$, the Hausdorff dimension of limit sets $\Lambda_{\Gamma_{L_t}}$ and $\Lambda_{\Gamma'_{L_t}}$ are both strictly smaller than one. 

However, these two loops are not homotopically equivalent and the homotopy classes of these loops in $\mathcal{S}ch_{<1}$ are distinguished by their monodromy given in Table 1 in \cite{He}. Here $\mathcal{S}ch_{<1}$ denotes the space of Schottky groups whose limit set has Hausdorff dimension smaller than one.

Now let $\iota : \PSL(2,\bbC) \to \PSL(3,\bbC)$ be the unique irreducible representation, i.e.
\[ \begin{bmatrix}
a & b \\
c & d 
\end{bmatrix}
\mapsto
\begin{bmatrix}
a^2 & ac & c^2 \\
2ab & ad+bc & 2cd \\
b^2 & bd & d^2
\end{bmatrix} \]

Consider the compositions $\iota \circ \Gamma_{L_t}$ and $\iota \circ \Gamma'_{L_t}, t \in [0,1]$.
\begin{lem}
 $\gamma = \iota \circ \Gamma_{L_t}$ and $\gamma' =\iota \circ \Gamma'_{L_t}, t \in [0,1]$ are two loops  in $\mathcal{S}_{< 1}$. 
\end{lem}
\begin{proof}
First we note that $\iota \circ \Gamma_{L_t}$ and $\iota \circ \Gamma'_{L_t}, t \in [0,1]$ are two loops since $\Gamma_{L_0} = \Gamma_{L_1}$, $\iota \circ \Gamma_{L_0} = \iota \circ \Gamma_{L_1}$. Moreover, $\iota \circ \Gamma_{L_t}$ is continuous in $t$. Same holds for $\iota \circ \Gamma'_{L_t}$.

We show that the two loops are in $\mathcal{S}_{<1}$. If $\rho_0 : \partial\pi_1\Sigma \to \PSL(2,\bbC)$ is a Schottky representation, then $\rho =\iota \circ \rho_0$ is $n$-Fuchsian. In particular, it is Anosov with respect to the minimal parabolic subgroup. Therefore, $\rho$ is $(1,1,2)$-hyperconvex. Moreover, $\zeta_{\rho}^1(\partial\pi_1\Sigma)$ has Hausdorff dimension smaller than one since the map $\iota$ is algebraic so that Hausdorff dimension remains the same.
\end{proof}

Moreover, numerical calculations suggest that $\gamma$ and $\gamma'$ are not homotopic and their monodromies are summarized in the following table.
\begin{table}[h] \label{table_mon}
	\begin{center}
		\begin{tabular}{ |c|c|c| } 
			\hline
			Words & Monodromy along $\gamma$\text{~~} & Monodromy along $\gamma'$ \\ 
			\hline
			a & $4\pi$ & $20\pi$\\ 
			\hline
			b & $4\pi$ & $12\pi$\\  
			\hline
			A & $4\pi$ & $20\pi$\\ 
			\hline 
			B & $4\pi$ & $12\pi$\\
			\hline
			ab & $4\pi$ & $0$\\  
			\hline 
			AB & $4\pi$ & $0$\\  
			\hline 
			aB & $0$ & $4\pi$\\  
			\hline
			bA & $0$ & $4\pi$\\  
			\hline
			Total change & $24\pi$ & $72\pi$\\  
			\hline  
		\end{tabular}
		\vspace*{0.2cm}
		\caption{Monodromy for the loops $\gamma$ and $\gamma'$}
	\end{center}
\end{table}


\begin{thebibliography}{8}
	\bibitem{Bas} A. Basmajian, {\it The orthogonal spectrum of a hyperbolic manifold}, American Journal of Mathematics, {\bf 115} (1993), no. 5, 1139–1159.
	
	\bibitem{BPS} J. Bochi, R. Potrie and A. Sambarino, {\it Anosov representations and dominated splittings}, to appear in Jour. Europ. Math. Soc. Accepted in 2017.
	
	\bibitem{BCLS} M. Bridgeman, R. Canary, F. Labourie and A. Sambarino, {\it The Pressure metric for Anosov representations}, Geom. Funct. Anal. {\bf 25} (2015), 1089-1179.
	
	\bibitem{FGW} E. Falbel, A. Guilloux and P. Will, {\it Hilbert metric, beyond convexity}, arXiv:1804.05317.
	
	\bibitem{He} Y. M. He, {\it Basmajian-type identities and Hausdorff dimension of limit sets}, Ergodic Theory and Dynamical Systems, {\bf 38} (2018), 2224-2244.
	
	\bibitem{He1} Y. M. He, {\it Prime number theorems for Basmajian-type identities}, arXiv:1811.05367.

	\bibitem{Lab06} F. Labourie, {\it Anosov flows, surface groups and curves in projective space}, Inventiones Mathematicae {\bf 165} (2006), 51-114.
		
	\bibitem{Lab07} F. Labourie, {\it Cross ratios, surface groups and PSL$(n,\bbR)$ and diffeomorphisms of the circle}, Publ. Math. Inst. Hautes Etudes Sci. {\bf 106} (2007), 139-213.
	
	\bibitem{LabMc} F. Labourie and G. McShane, {\it Cross ratios and identities for higher Teichm\"uller-Thurston theory}, Duke Math. J. {\bf 149} (2009), 279-345.
	
	
	\bibitem{PSW} B. Pozzetti, A. Sambarino and A. Wienhard, {\it Conformality for a robust class of non-conformal attractors}, arXiv:1902.01303.
	
	\bibitem{Vlamis} N. Vlamis and A. Yarmola, {\it Basmajian's identity in higher Teichm\"uller-Thurston theory}, Journal of Topology, {\bf 10} (2017), 744-764.
	
	\bibitem{ZhangZimmer} T. Zhang and A. Zimmer, {\it Regularity of limit sets of Anosov representations}, https://arxiv.org/pdf/1903.11021.pdf
\end{thebibliography}
\end{document}